\documentclass{amsart}
\usepackage[utf8]{inputenc}

\usepackage{amsmath, amssymb, amsfonts, amsthm}
\usepackage{enumerate}
\usepackage{bbold}
\usepackage{color}
\usepackage{xcolor}
\usepackage{esvect}
\usepackage{mathtools}
\usepackage{graphicx}
\usepackage{longtable}
\usepackage{fullpage}
\usepackage{rotating}
\usepackage{blkarray}
\usepackage{tikz}
\usepackage{tikz-cd}
\usetikzlibrary{quotes}
\usepackage[boxsize = .3em]{ytableau}

\usepackage{comment}

\usepackage[
backend=biber,
style=numeric,
]{biblatex}
\theoremstyle{plain}
\newtheorem{thm}{Theorem}[section]
\newtheorem{lem}[thm]{Lemma}
\newtheorem{prop}[thm]{Proposition}
\newtheorem{cor}[thm]{Corollary}
\newtheorem{conj}[thm]{Conjecture}
\newtheorem{ass}[thm]{Assumption}
\newcommand{\att}[2]{\raisebox{-.5\height}{ \includegraphics[scale = #2]{diagrams/#1.pdf}}}
\theoremstyle{definition}
\newtheorem{defn}[thm]{Definition}

\newtheorem{rmk}[thm]{Remark}

\newtheoremstyle{named}{}{}{\itshape}{}{\bfseries}{.}{.5em}{#1 \thmnote{(#3) }}
\theoremstyle{named}

\newcommand{\ol}{\overline}

\newcommand{\N}{\mathbb{N}}
\newcommand{\Z}{\mathbb{Z}}

\newcommand{\C}{\mathbb{C}}

\newcommand{\CC}{\mathcal{C}}

\newcommand{\ii}{ i }

\DeclareMathOperator{\im}{im}

\DeclareMathOperator{\Hom}{Hom}

\DeclareMathOperator{\GPA}{GPA}
\DeclareMathOperator{\End}{End}
\DeclareMathOperator{\Vecc}{Vec}

\DeclareMathOperator{\Ab}{Ab}

\newcommand{\red}[1]{{\color{red}{#1}}}

\title{A graph planar algebra approach to near-group categories}
\author{Cain Edie-Michell and Caleb Kennedy Hill}
\date{}

\begin{document}

\begin{abstract}
In this note we investigate the construction of near-group fusion categories via building explicit representations of them in certain graph planar algebras. To achieve this we first give a presentation for a cyclic near-group fusion category. Our presentation is ``subfactor-centric'' in the sense that we use a Q-system as one of our generating morphisms. Using this presentation we then obtain a rational system of equations, a solution of which corresponds to a faithful embedding of a near-group fusion category into a certain 2-coloured graph planar algebra. We end the note by providing solutions to these systems of equations for the odd cyclic groups up to order 13. This gives an alternate construction to the Evans-Gannon construction of the corresponding near-group fusion categories.       
\end{abstract}

\maketitle

\section{Introduction}

A fusion category is called a near-group fusion category if it has only one non-invertible simple object \cite{Siehler}.
In terms of the underlying fusion ring, this class of categories represents the simplest fusion categories which are not $\Vecc(G)$
for some finite group $G$. If $\CC$ is a near-group fusion category 
its invertible simple objects form a group $\CC^{\text{Inv}}$;
we will usually denote by just $G$.
If $\rho$ is the single non-invertible simple object, then the following fusion rules hold for some $N\in \mathbb{N}$:
\begin{equation*}
     g\otimes h = gh,\quad  \forall g,h\in G, \qquad
    g\otimes \rho = \rho\otimes g = \rho, \quad \forall g\in G,  \qquad
  \rho\otimes \rho = N\cdot \rho\oplus \bigoplus_{g\in G} g.
\end{equation*}
Such a category with these fusion rules is typically called a $G + N$ near-group fusion category.

Despite the simplicity of the fusion rules of a near-group fusion category, their categorical structure is surprisingly rich and difficult to understand. Due to work of Evans-Gannon \cite{E-G} and Izumi \cite{Izumi} the parameters $G$ and $N$ are constrained. There is a fundamental bifurcation. If the Frobenius-Perron dimension of $\rho$ is rational, then construction and classification is feasible and has been completed. In the case that the dimension is irrational, we have that the group $G$ is abelian and the multiplicity parameter $N$ is a multiple of $|G|$. In this case construction is much more difficult. In fact, in the irrational case only finitely many examples of near-group categories are known to exist with $N\neq 0$. The following conjecture is due to Izumi \cite{IzumiOld} and Evans-Gannon.
\begin{conj}
     For all $N\in \mathbb{N}$, there exists a $\mathbb{Z}_N + N$ near-group fusion category.
 \end{conj}

 \begin{rmk}
Just before this note was initially posted to the arXiv, there were several developments regarding the above conjecture. In \cite{gsy} a $\mathbb{Z}_{N^2} + N^2$ near-group fusion category was constructed for all $N\in \mathbb{N}$. In \cite{other-paper}, a $\mathbb{Z}_{N}\times \mathbb{Z}_{M} + NM$ near-group fusion category was constructed for all $N,M \in \mathbb{N}$. This later result fully resolves (and improves on) the above conjecture. Note that in both works, the categories constructed are not verified to be unitary. One small saving grace for this note is that our solutions are explicitly algebraic, and unitarity is explicitly verified by studying the minimal polynomial of the field extension the solution lives in. Another point of difference of this note is that our classification data is fundamentally different from Izumis classification data. i.e. the solutions of the above papers are not solutions to our equations. Though clearly the scope of this note is weakened with these new developments.
 \end{rmk}

 The construction methods of \cite{E-G,Izumi} work by representing a $G + N$ near-group category on the tensor category of endomorphisms of the Cuntz algebra $\mathcal{O}_{|G| + N}$. This method is essentially equivalent to solving for the 6-j symbols of the near-group category. In this note we explore an alternate construction method via representing the near-group categories on a variant of Jones's graph planar algebra \cite{OGGPA}.

 From the point of view of this note, it is best of think of the graph planar algebra as a nice monoidal subcategory of the monoidal category of endofunctors of a semisimple category $\mathcal{M}$. If the semisimple category $\mathcal{M}$ has rank $n$, then endofunctors on $\mathcal{M}$ are uniquely determined by a graph $\Gamma$ on $n$ vertices. The graph planar algebra $\operatorname{GPA}(\Gamma)$ can then be defined as the full monoidal subcategory of $\operatorname{End}(\mathcal{M})$ generated by tensor powers of the endofunctor associated to $\Gamma$. This monoidal subcategory is then relabelled so $\operatorname{GPA}(\Gamma)$ is defined purely in terms of graph theoretic data. Due to the connection with $\operatorname{End}(\mathcal{M})$, the graph planar algebra is universal in the sense that any sufficiently nice tensor category can be represented on $\operatorname{GPA}(\Gamma)$ for some choice of graph $\Gamma$. We direct the reader to \cite{ETMod, EH} for further details on this connection.

 In this note we introduce the multi-colour graph planar algebra, associated to a coloured graph. This definition is certainly known as folklore to experts (see \cite{Emily} for a close definition). The motivating logic is the same as the standard graph planar algebra, except now we consider the full monoidal subcategory of $\operatorname{End}(\mathcal{M})$ generated by tensor powers of a collection of graphs (which we index by ``colours''). We direct the reader to Subsection~\ref{subsec:gpa} for the precise definition.

 The construction strategy of this paper is to show the existence of $\mathbb{Z}_N + N$ near-group fusion categories via representing them on the graph planar algebra of a two coloured graph. One colour will correspond to the $\rho$ object, and the other will correspond to a group generator of the $\mathbb{Z}_N$ subcategory.

 As with other algebraic objects, the easiest way to construct a functor out of a tensor category $\mathcal{C}$ is via giving a generators and relations presentation for $\mathcal{C}$ (i.e. establishing a universal property). This is the first order of business in this note, which we achieve in Section~\ref{sec:skein}. Our presentation for a $\mathbb{Z}_N + N$ near-group fusion category has two generating objects, $\rho$ and the group generator $g$ of $\mathbb{Z}_N$, and three generating morphisms. The first generating morphism is the isomorphism $g^{\otimes N} \to \mathbf{1}$. The second is the multiplication map for a Q-system structure (see \cite{Q-sys}) on $\mathbf{1} \oplus \rho$. By results of Izumi such a structure always exists provided the near-group category exists, and so we lose no generality by enforcing its existence. The final generator is an isomorphism $\rho \otimes g \to \rho$. We are able to pin down relations for these generators, which depend on $N$ free parameters. With non-triviality and unitary assumptions on this generators and relations category (which will come for free with an embedding into the unitary graph planar algebra), we prove that after semisimplification and Cauchy completion (see Subsection~\ref{sub:Cpre}), we obtain a $\mathbb{Z}_N + N$ near-group fusion category.

 Given our general presentation for a $\mathbb{Z}_N + N$ near-group fusion category $\mathcal{C}$, we now move on to representing such a category in a multi-coloured graph planar algebra. This requires a choice of 2-coloured graph. The immediate choice for a 2-coloured graph is the fusion graph of $\mathcal{C}$, with one colour the fusion with $\rho$, and the other colour the fusion with the group generator $g$. This graph has the downside that there is edge multiplicity $N$ coming from the fusion space $\dim\Hom_{\mathcal{C}}(\rho \to \rho \otimes \rho) = N$. This manifests in a large gauge group ($(\mathbb{C}^\times)^N \times M_N(\mathbb{C})$) for our solution space. This would make finding the embedding difficult, and would force us to make unnatural choices to pin down the gauge. Instead we choose to work with a module fusion graph for $\mathcal{C}$ corresponding to a non-trivial module category over $\mathcal{C}$. In the case of $N=5$ the corresponding 2-coloured graph is
\begin{equation}\label{eq:exGraph}
    \begin{tikzpicture}[scale=0.5, every node/.style={circle, draw, minimum size=5mm}]
        \node (1) at (90:2)   {*};
        \node (2) at (30:2)   {0};
        \node (3) at (-30:2)  {1};
        \node (4) at (-90:2)  {2};
        \node (5) at (-150:2) {3};
        \node (6) at (150:2)  {4};

        \foreach \i in {1,...,6} {
            \foreach \j in {1,...,6} {
            \ifnum\i<\j
                \draw (\i) -- (\j);
            \fi
            }
        }

        \draw (2) edge[loop, out=90,  in=0, looseness=5]  (2);
        \draw (3) edge[loop, out=-90, in=0, looseness=5]  (3);
        \draw (4) edge[loop, out=-135, in=-45, looseness=5] (4);
        \draw (5) edge[loop, out=180,in=-90, looseness=5] (5);
        \draw (6) edge[loop, out=180, in=90, looseness=5] (6);
    
        \draw (1) edge[loop, red,->,looseness=5] (1);
         \draw (2) edge[red,->,out = -20, in = 20] (3);
         \draw (3) edge[red,->,out = -110, in = -10] (4);
          \draw (4) edge[red,->,out = 190, in = -250+180] (5);
          \draw (5) edge[red,->,out = 160, in = -160] (6);
           \draw (6) edge[red,->,out = 15, in = 165] (2);
    \end{tikzpicture}
\end{equation}
The hom spaces in the corresponding graph planar algebra are considerably smaller, and our solution space gauge group is now $(\mathbb{C}^\times)^{N + \frac{N(N+1)}{2}} $. Working in this family of graph planar algebras we are able to obtain for all $N$, embeddings for the isomorphism $g^{\otimes N} \to \mathbf{1}$, and embeddings for the multiplication map generator. The latter has (to us) surprisingly nice embedding coefficients, which are square roots of products of certain quantum integers. This leaves the isomorphism $\rho \otimes g\to \rho$ to be embedded. We enforce a natural projective symmetry on our solution space (which completely fixes a gauge choice). This reduces the solution space to $N$ unimodular complex scalars. The main result of this section and note is then a system of rational (or polynomial with complex conjugates) system of equations, a solution to which gives a embedding of our presentation for a near-group fusion category $\mathcal{C}$ into our graph planar algebra. Such a solution in turn proves existence of a $\mathbb{Z}_N + N$ near-group fusion category. This main result is Corollary~\ref{cor:main}. As our construction is manifestly unitary, and the Q-system is baked into our presentation, solutions to our system of equations also gives rise to $2_1^{\mathbb{Z}_N}1$ subfactors in the sense of Izumi \cite[Section 11]{Izumi}.

We finish this note by obtaining solutions for our system of equations for odd $N$ up to $N=13$, and hence obtaining new constructions of the corresponding near-group fusion categories. Our solutions live in degree 2 and 4 extensions over the cyclotomic fields $\mathbb{Q}( e^{2\pi i \frac{1}{N}},e^{2\pi i \frac{1}{N+4}}, i)$. This provides an alternate proof of the existence of these near-group categories from the work of Evans-Gannon \cite{E-G}. Our solutions are clearly much too nice to be completely random, and there appears to be some structure to them. However our human eyes and brains were unable to find any sort of general pattern.

\subsection*{Acknowledgements}
CE was supported by NSF DMS grant 2400089. CE would like to thank Noah Snyder for informing them about the relationship between Frobenius algebras in tensor categories, and $SO(3)_q$ subcategories, and for informing them of the work of \cite{other-paper}.  
\section{Preliminaries}

We refer the reader to \cite{Book} for the general background on tensor categories.

\subsection{ Categorical Preliminaries}\label{sub:Cpre}
In this section we review the required categorical notions required for this note. We begin with the Cauchy completion of a tensor category. Informally this construction adds direct sums and subobjects to a tensor category. The construction to add subobjects is the idempotent completion.

\begin{defn}
    Let $\mathcal{C}$ be a pivotal tensor category. We define $\operatorname{Idem}(\mathcal{C})$ as the category whose objects are pairs $(X, p)$ where $X\in \mathcal{C}$, and $p \in \End_\mathcal{C}(X\to X)$ is a projection. The morphisms are given By
    \[  \Hom_{\operatorname{Idem}(\mathcal{C})} ( (X, p)\to (Y, q)):= \{  f\in \Hom_\mathcal{C}(X\to Y) : p \circ f = f\circ q\}.                      \]
    The tensor product and pivotal structure are inherited from the category $\mathcal{C}$ in the natural way.
\end{defn}

To add direct sums, we define the additive completion.

\begin{defn}
    Let $\mathcal{C}$ be a pivotal tensor category. We define $\operatorname{Add}(\mathcal{C})$ as the category with object formal direct sums
    \[  \bigoplus_i X_i   \]
    where $X_i\in \mathcal{C}$. The morphisms are given by 
       \[  \Hom_{\operatorname{Add}(\mathcal{C})}\left(\bigoplus_{i=1}^n X_i \to \bigoplus_{j=1}^m Y_j\right) := \left\{ \begin{blockarray}{ccccc}
&Y_1  & Y_2 &\cdots & Y_m    \\
\begin{block}{c[cccc]}
X_1 & f_{1,1} & f_{1,2} & \cdots & f_{1,m} \\
X_2 & f_{2,1} & f_{2,2} & \cdots & f_{2,m} \\
\vdots & \vdots & \vdots& \ddots & \vdots \\
X_n &f_{n,1} & f_{n,2} & \cdots & f_{n,m} \\
\end{block}
\end{blockarray}: X_i, Y_j \in \mathcal{C} , f_{i,j}\in \Hom_\mathcal{C}(X_i \to Y_j)\right\}.\]
The composition of morphisms is defined by matrix multiplication. The tensor product of morphisms is defined by the Kronecker product. The pivotal structure is inherited from $\mathcal{C}$.
\end{defn}

We define the Cauchy completion as the combination of these two constructions.

\begin{defn}
 Let $\mathcal{C}$ be a pivotal tensor category. We define $\operatorname{Ab}(\mathcal{C}):= \operatorname{Add}(\operatorname{Idem}(\mathcal{C}))$.
\end{defn}

The Cauchy completion satisfies a nice universal property that we will make use of in this note.

\begin{prop}\label{prop:uni}\cite[Equations (7) and (9)]{comes}
    Let $\mathcal{C}$ and $\mathcal{D}$ be pivotal tensor categories, with $\mathcal{D}$ Cauchy complete. Then there is an equivalence of functor categories
    \[   \operatorname{Fun}_{\otimes}(\operatorname{Ab}(\mathcal{C}) \to \mathcal{D}   )\simeq \operatorname{Fun}_{\otimes}(\mathcal{C} \to \mathcal{D}   ).    \]
    The foward functor sends $\mathcal{F}
\mapsto \iota_\mathcal{C} \circ \mathcal{F}$ where $\iota_\mathcal{C}: \mathcal{C}\to \operatorname{Ab}(\mathcal{C})$ is the inclusion functor.
The backwards functor sends $\mathcal{G} \mapsto \operatorname{Ab}(\mathcal{G})$ where $\operatorname{Ab}(\mathcal{G})[X, p] := \im(  \mathcal{G}(p))$.
\end{prop}

The following theorem is well-known. However we were unable to find a proof in the literature. We include one here for completeness.
\begin{prop}\label{prop:ssFull}
    Let $\mathcal{C},\mathcal{D}$ be categories with $\mathcal{D}$ finitely semisimple, and $\mathcal{F}: \mathcal{C}\to \mathcal{D}$ a faithful functor with the property that for all simple objects $Y\in \mathcal{D}$, there exists an object $X_Y\in \mathcal{C}$, and a projection $p_Y \in \End_{\mathcal{C}}(X_Y )$ such that $\mathcal{F}(p_Y)$ projects onto $Y$. Then $\mathcal{F}$ is full.
\end{prop}
\begin{proof}
    By Proposition~\ref{prop:uni}, we have that $\mathcal{F}$ extends to an additive functor $\operatorname{Ab}(\mathcal{F}): \operatorname{Ab}(\mathcal{C}) \to \mathcal{D}$. For an object $(X, p)\in \operatorname{Ab}(\mathcal{C})$ we have by construction that $ \operatorname{Ab}(\mathcal{F})( X, p)  = \im(   \mathcal{F}(p) )   $.

    We claim $ \operatorname{Ab}(\mathcal{F})$ is essentially surjective. As $\mathcal{D}$ is semisimple, and $ \operatorname{Ab}(\mathcal{F})$ is additive, it suffices to verify essentially surjective on simples. Let $Y \in \mathcal{D}$ be a simple object. By assumption there exists an object $X_Y \in \mathcal{C}$, and a projection $p_Y \in \operatorname{End}_\mathcal{C}(X_Y)$ such that $\im( \mathcal{F}(X_Y, p_Y)) \cong Y$. Hence $ \operatorname{Ab}(\mathcal{F})( X_Y, p_Y) \cong Y$, and so $ \operatorname{Ab}(\mathcal{F})$ is essentially surjective.

    The claim that $ \operatorname{Ab}(\mathcal{F})$ is full then follows from the general result that a faithful essentially surjective additive functor into a finitely semisimple category is automatically full.
\end{proof}

Finally we define the semisimple quotient.

\begin{defn}
Let $\mathcal{C}$ be a spherical tensor category. We define the negligible ideal.
\[  \operatorname{Neg}(\mathcal{C}):= \{ f\in \Hom_{\mathcal{C}}(X\to Y) : \operatorname{tr}(f \circ g) = 0 \text{ for all } g \in \Hom_{\mathcal{C}}(Y\to X)      \} .   \]
\end{defn}

The negligible ideal is a tensor ideal, and the quotient is again a spherical tensor category. This quotient is called the semisimple quotient (even though it is not always semisimple).

\begin{defn}
    Let $\mathcal{C}$ be a spherical tensor category. We define
    \[ \overline{\mathcal{C}}:=  \mathcal{C}/   \operatorname{Neg}(\mathcal{C}).   \]
\end{defn}
We will use the following standard result in this note.
\begin{prop}\label{prop:descent}\cite[Proposition 2.39]{EM-S}
    Let $\mathcal{C}$ be a spherical category, let $\mathcal{D}$ be a unitary tensor category, and let $\mathcal{F}:\mathcal{C}\to \mathcal{D}$ be a pivotal functor. Then there exists a faithful pivotal functor $\overline{\mathcal{F}}: \overline{\mathcal{C}}\to \mathcal{D}$ such that the following diagram commutes
      \[ \begin{tikzcd}
\mathcal{C} \arrow[d] \arrow[r, "\mathcal{F}"] & \mathcal{D} \\
\overline{\mathcal{C}} \arrow[ur, "\overline{\mathcal{F}}"]& 
\end{tikzcd}\]
where $\mathcal{C}\to \overline{\mathcal{C}}$ is the semisimplification functor. In particular, $\overline{\mathcal{C}}$ is a unitary category.
\end{prop}
\subsection{Near-Group Fusion Categories}

In this subsection we introduce near-group fusion categories. These were initially defined by Siehler \cite{Siehler}, and studied in detail by Izumi \cite{Izumi} and Gannon and Evans \cite{E-G}. We first define a near-group fusion ring.

\begin{defn}
Let $G$ be a finite group, $\rho$ be a symbol, and $N\in \mathbb{N}$. We define the fusion ring $R(G,n)$ as the ring with $\mathbb{Z}$-basis $G \cup \{\rho\}$, and multiplication defined by the rules:
\begin{align*}
g \cdot h &=     g\circ h \qquad &&\text{ for all $g,h \in G$}\\
g \cdot \rho = \rho \cdot g &=     \rho \qquad &&\text{ for all $g \in G$}\\
\rho \cdot \rho &= \sum_{g \in G} g + N \rho.
\end{align*}
\end{defn}
We then define a near-group fusion category as a categorification of the above rings.
\begin{defn}
Let $G$ be a finite group, and $N\in \mathbb{N}$. We say a unitary fusion category $\mathcal{C}$ is a ($G+N$) near-group fusion category if there is a based isomorphism
\[  K_0(\mathcal{C}) \cong R(G, N).       \]
Here $ K_0(\mathcal{C})$ is the Grothendieck ring of $\mathcal{C}$.
\end{defn}

We now present some known results on the structure of near-group fusion categories.

\begin{thm}\cite[Theorem 1.1]{Izumi}
Let $G$ be a finite group, $N\in \mathbb{N}$, and $\mathcal{C}$ a $G+N$ near-group fusion category. Then either
\begin{enumerate}[a)]
    \item $N = |G|-1$, and $G$ is abelian, or
    \item $N = 2^a$ for some $a\in \mathbb{N}$ and $G$ is an extra-special 2-group of order $2^{2a+1}$, or
    \item $N = k |G|$ for some $k\in \mathbb{N}$, and $G$ is abelian.
\end{enumerate}
\end{thm}
 For cases a) and b), the near-group categories have been classified \cite{EGO,Izumi}, and so attention is currently placed on the near-group fusion categories in case c). In this note we focus on $\mathbb{Z}_N + N$ near-group fusion categories. As mentioned in the introduction, these categories are conjectured to exist for all $N$ (see \cite{An} for examples of non-cyclic abelian groups $G$ where there are no $G + N$ near group categories).

For all near-group fusion categories we have the following.

\begin{prop}\label{lem:subcat1}\cite{Izumi}
Let $G$ be a finite group, $N\in \mathbb{N}$, and $\mathcal{C}$ a $G + N$ near-group fusion category. Then the monoidal subcategory
\[  \mathcal{C}^{\text{Inv}}:= \langle g \in \mathcal{C} : g  \text{ is invertible} \rangle   \]
is monoidally equivalent to $\operatorname{Vec}(G)$, the category of $G$-graded vector spaces with trivial 3-cocycle.
\end{prop}
 
In the specific case that $N = |G|$ we also have the following which is also due to Izumi.

\begin{prop}\label{lem:subcat2}\cite[Theorem 11.2]{Izumi}
    Let $G$ be a finite abelian group, and $\mathcal{C}$ a $G + |G|$ near-group fusion category. Let $\rho \in \mathcal{C}$ be the unique (up to isomorphism) non-invertible simple object. Then $\mathbf{1} \oplus \rho$ has the structure of a Q-system. i.e. a special $C^*$-Frobenius algebra.
\end{prop}
\begin{proof}
In \cite[Theorem 11.2]{Izumi} Izumi constructs a semisimple indecomposable $\mathcal{C}$-module category $\mathcal{M}$, with an object $\iota \in \mathcal{M}$ satisfying $\underline{\operatorname{End}}(\iota) \cong \mathbf{1} \oplus \rho$. It follows from \cite[Theorem 1]{OstMod} that $\mathbf{1} \oplus \rho$ is a connected separable algebra object, and from \cite{Whale} we get that $\mathbf{1} \oplus \rho$ has the structure of a Q-system.
\end{proof}

\subsection{Multicolour Graph Planar Algebras}\label{subsec:gpa}

We now define the multicolour graph planar algebra. This object will be the key tool for constructing near-group fusion categories in this note. The idea behind the multicolour graph planar algebra is certainly not new. The graph planar algebra of \cite[Section 2.2]{Emily} is essentially what we define below, however we also include rigidity maps in our definition. Our definition is also very similar to the oriented graph planar algebra of \cite{dan}, which is essentially the one colour variation of our construction. 

To explain the key motivation behind the definition of a multicolour graph planar algebra, we outline the connection between the graph planar algebra, and the endofunctor category of a semisimple category. Recall that any finitely semisimple category $\mathcal{M}$ is equivalent to $\operatorname{Vec}^{\oplus n}$, where $n$ is the rank of the category. Up to natural isomorphism, an endofunctor on $\mathcal{M}$ is entirely determined by a $n\times n$ matrix, whose entries are in $\mathbb{N}_{\geq 0}$ (which encodes where the simples of $\mathcal{M}$ get mapped to). Such a matrix is equivalent data to a graph $\Gamma$ with $n$ vertices. The usual graph planar algebra $\operatorname{GPA}(\Gamma)$ when $\Gamma$ can then be defined as the rigid monoidal subcategory of $\operatorname{End}(\mathcal{M})$ generated by tensor powers of the object $\Gamma$.

The logic behind the multicolour graph planar algebra is similar to the above discussion. The difference now is that we have a collection of graphs $\{ \Gamma_i \}$, which we will think of as being labelled by colours. The multicolour graph planar algebra $\operatorname{GPA}(\{ \Gamma_i \})$ will then be (up to relabelling), the rigid monoidal subcategory $\operatorname{End}(\mathcal{M})$ generated by these objects.

With this motivation in mind we now present the definitions. We first define the concept of a coloured graph.

\begin{defn}
    Let $\mathcal{I}$ be a set. An $\mathcal{I}$-coloured graph is a pair $(V, \{E_i, s_i, t_i  : i \in \mathcal{I}\} )$ where $V$ is a set of vertices, each $E_i$ is a set of edges, and $s_i, t_i: E_i \to V$ are functions. 
\end{defn}
We should think of an $\mathcal{I}$-coloured graph as $| \mathcal{I}|$ individual graphs overlayed on each other on the same vertex set $V$. The labels $\mathcal{I}$ which distinguish each layer we think of as colours.
\begin{defn}
Let $\mathcal{I}$ be a set, and $\Gamma = (V, \{E_i, s_i, t_i  : i \in \mathcal{I}\} )$ an $\mathcal{I}$-coloured graph. We define $\Gamma_i := (V, E_i,s_i,t_i)$. 
\end{defn}
Note that $\Gamma_i$ is a graph in the traditional sense.

We can also define paths in a coloured graph.

\begin{defn}
    Let $\mathcal{I}$ be a set, and $\Gamma= (V, \{E_i, s_i, t_i  : i \in \mathcal{I}\} )$ an $\mathcal{I}$-coloured graph. For a word $t:= i_1i_2\cdots i_m$ with $i_l\in \mathcal{I}$ we say a coloured path of type $t$ in $\Gamma$ is a ordered tuple of edges $p=(e_1, e_2, \cdots, e_{m})$ such that $e_j \in E_{i_j}$, and such that $t_{i_{j-1}}(e_{j-1}) = s_{i_j}(e_j)$ for all $1< j \leq m$. We will write $s(p) := s(e_1)$ and $t(p) := t(e_m)$.
\end{defn}

To give our upcoming definition of the multicolour graph planar algebra a rigidity structure, we will need to define a rigidity structure on a multicolour graph. This will be a condition that the set of coloured subgraphs is closed under the reversal of edges, and a choice (in the case that the graph has edge multiplicity) of identification between the edges in a graph and its reversal.

\begin{defn}
Let $\mathcal{I}$ be a set, and $\Gamma = (V, \{E_i, s_i, t_i  : i \in \mathcal{I}\} )$ an $\mathcal{I}$-coloured graph. A rigidity structure on $\Gamma$ is a pair $(* ,\{ \overline{\cdot}^i : i\in I\})$ such that $* : \mathcal{I} \to \mathcal{I}$ is an involution, and each $ \overline{\cdot}^i: E_i \to E_{i^*}$ is a bijection such that 
\[ t_i(e) = s_{i^*}(\overline{e}^i) \quad \text{ and }\quad s_i(e) = t_{i^*}(\overline{e}^i) \quad \text{ for all } e\in E_i,  \]
and with the property that $\overline{\cdot}^i$ and $\overline{\cdot}^{i^*}$ are inverse functions.
\end{defn}
In practice we will drop the superscript on the map $\overline{\cdot}^i$, as the edge the map is applied to will remove ambiguity.

We are now set to define the multicolour graph planar algebra.

\begin{defn}
Let $\mathcal{I}$ be a set, and $\Gamma = (V, \{E_i, s_i, t_i  : i \in \mathcal{I}\} )$ an $\mathcal{I}$-coloured graph. We define the monoidal category $\operatorname{GPA}(\Gamma)$ as the category whose objects are finite strings $i_1i_2 \cdots i_m$ in the alphabet $\mathcal{I}$. Let $s =i_1i_2\cdots i_m$ and $t= j_1j_2 \cdots j_l$ be objects. We define the morphisms $s\to t$ as
\[   \Hom( s \to  t):= \operatorname{span}_{\mathbb{C}}\{(p,q) : \text{$p$ is path of type $s$, $q$ is a path of type $t$, $s(p) = s(q)$, and $t(p) = t(q)$}    \}                  \]
The composition $ \circ: \Hom( s \to  t) \times \Hom( t \to  r) \to \Hom( s \to  r) $ is defined as the $\mathbb{C}$-linear extension of the map
\[     (p,q)\circ (p' ,q'):= \delta_{q = p'}(p,q').            \]
Note that the identity morphism on $i$ is then $\sum_{e \in \Gamma_i}(e,e)$.

The tensor product of objects is concatenation of strings. Hence the tensor unit is the empty string $\emptyset$. The tensor product on morphisms $ \otimes: \Hom( r \to  s) \times \Hom( t \to  u) \to \Hom( rt \to  su) $ is defined as the $\mathbb{C}$-linear extension of the map
\[     (p,q)\circ (p' ,q'):= \delta_{t(p) = s(p')}\delta_{t(q) = s(q')}(p,q').            \]
We define a $\dag$-structure on $\operatorname{GPA}(\Gamma)$ as the conjugate-linear extension of the map
\[ (p,q)^\dag := (q,p).  \]

Let $(* ,\{ \overline{\cdot}^i : i\in I\})$ be a rigidity structure on $\Gamma$, and for each colour $i \in \mathcal{I}$ let $\lambda_i$ be a positive Frobenius-Perron eigenvector of $\Gamma_i$. Then we define a morphism
\[ \operatorname{ev}^{\lambda_i}_{i}:= \sum_{e\in E_i}  \sqrt{ \frac{ \lambda_i(  t_i(e))}{ \lambda_i(  s_i(e)) }    }  ( (e,\overline{e}), s_i(e)) \in \Hom(  i i^* \to \emptyset )        \]
We define $\operatorname{coev}^{\lambda_i}_{i}:= (\operatorname{ev}^{\lambda_i}_{i^*})^\dag \in \Hom(   \emptyset \to i^* i ) $. 
\end{defn}
From this definition it follows that $\operatorname{GPA}(\Gamma)$ is a unitary tensor category (via the canonical faithful embedding into $\operatorname{Hilb}$). Further, the maps $\operatorname{ev}^{\lambda_i}_{i}$ and $\operatorname{coev}^{\lambda_i}_{i}$ equip $\operatorname{GPA}(\Gamma)$ as a rigid tensor category, where it follows that duality on objects is
\[    (i_1i_2\cdots i_m)^* =   i_m^*      \cdots  i_2^{*}i_1^{*}.     \]
Finally a direct computation shows that the identity map is a natural isomorphism $** \to \operatorname{Id}_{\operatorname{GPA}(\Gamma)}$. Hence $\operatorname{GPA}(\Gamma)$ has the structure of a pivotal category.

Finally in the case that $i = i^*$ (i.e $\Gamma_i$ is undirected) we have that $i \cong i^*$ in $GPA(\Gamma)$. A choice for this isomorphism is the map
\[   \sum_{e \in E_i} (e, \overline{e})  \in \Hom(  i \to i^*).                                    \]
Direct computation gives that this isomorphism is symmetrically self dual with respect to our pivotal structure. 
\section{Skein Theory for Cyclic Near-group Categories}\label{sec:skein}
In this section we develop skein theory for near-group categories. We focus on the case that the group $G$ is cyclic of order $N$, though our graphical techniques will work in the general case with slight modifications. The main definition of this section is the following skein category.
\begin{defn}\label{def:NGskein}
    Let $N \in \mathbb{N}$, $\omega \in \mathbb{C}$, and $\vec{r}\in \mathbb{C}^N$. We define $\delta := \frac{N + \sqrt{N(N+4)}}{2}$. We define $\mathcal{C}(\mathbb{Z}_N, \omega, \vec{r})$ as the pivotal $\dag$-category generated by the objects $\rho$ and $g$, with $\rho$ symmetrically self-dual, and the morphisms
    \[ \att{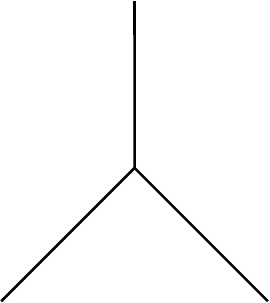}{.4} \in \Hom(\rho\otimes \rho \to \rho) ,\qquad \qquad \att{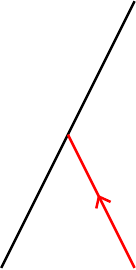}{.4} \in \Hom(\rho\otimes g \to \rho),\qquad \qquad \att{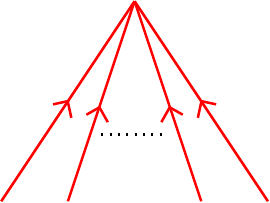}{.4} \in \Hom( g^{\otimes N} \to \mathbf{1}),         \]
    and the following relations:
    
$SO(3)$ Relations:
\[\att{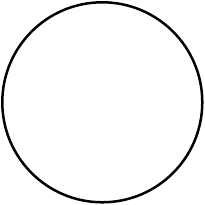}{.25} = \delta,\quad \att{T}{.25} = \att{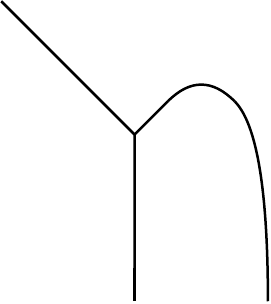}{.25},\quad \att{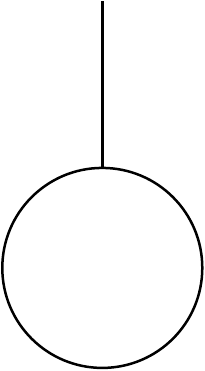}{.25}=0,\quad \att{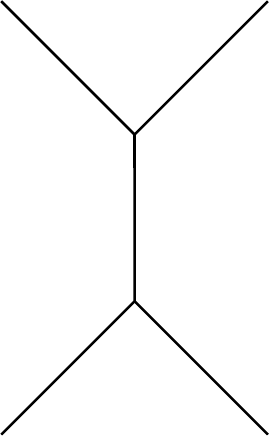}{.2} -\att{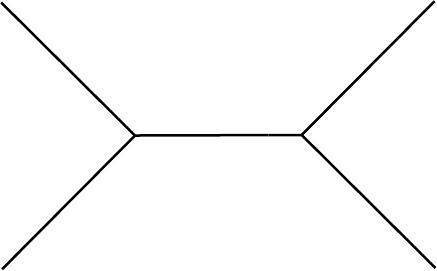}{.2} = \frac{1}{\delta-1}\left( \att{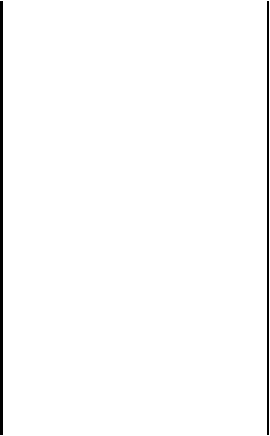}{.2}\quad -\quad \att{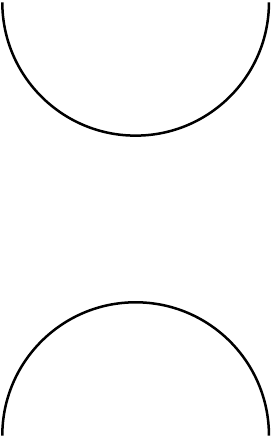}{.2} \right)  \]
$\operatorname{Vec}(\mathbb{Z}_N)$ Relations:
\[ \att{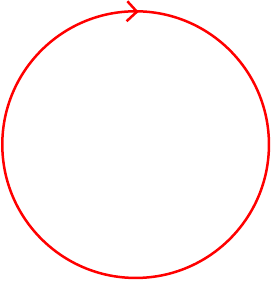}{.3} = 1 ,\qquad  \att{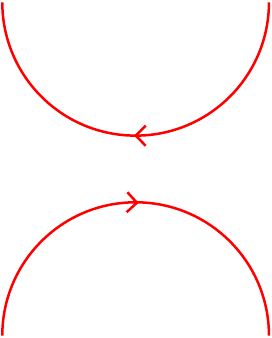}{.3} = \att{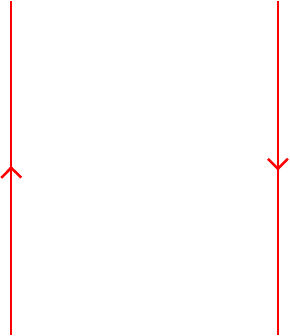}{.3},\qquad\att{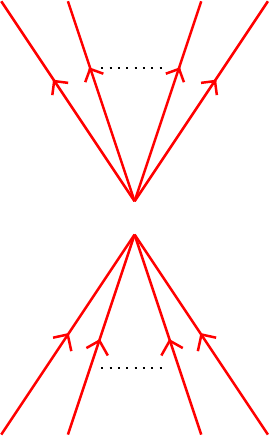}{.3}  = \att{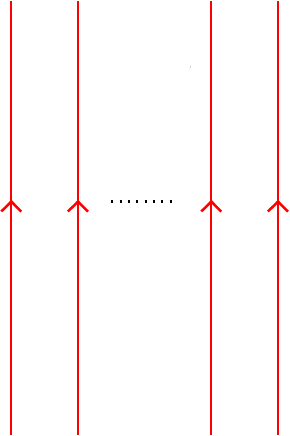}{.3}, \qquad  \att{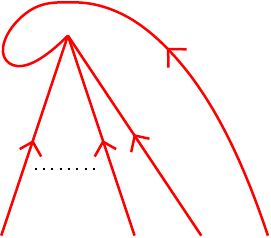}{.3} = \att{ggen}{.3} \]
Mixed  Relations:
\[\att{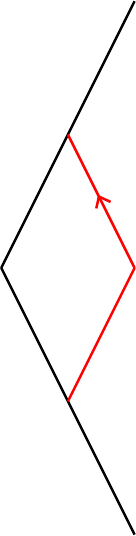}{.3} =  \att{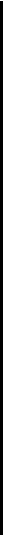}{.3},\qquad  \att{gMix1}{.3} = (-1)^{N+1} \att{gMix2}{.3},\qquad  \att{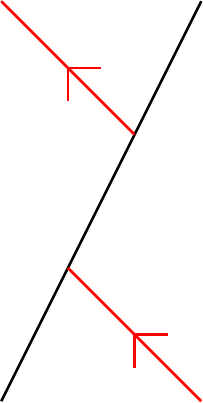}{.3} = \omega   \att{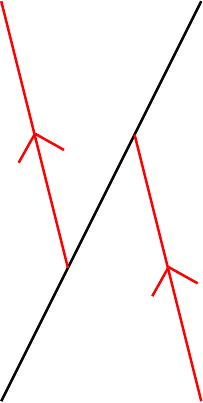}{.3} ,\qquad  \att{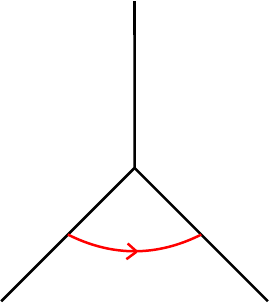}{.3} = \sum_{i = 0}^{N-1}   \vec{r}_i \att{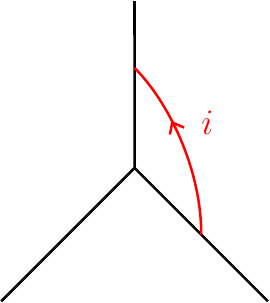}{.3},\]
\[ \att{schur0}{.3}\quad = \quad 0,\qquad     \att{schur1}{.3} \quad = \quad 0 ,\qquad \text{ for all } 1\leq i < N \]
\end{defn}

In this definition and throughout the remainder of the paper we use the shorthand
\[ \att{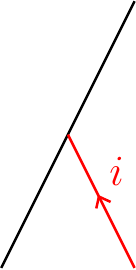}{.4} \quad:= \quad \att{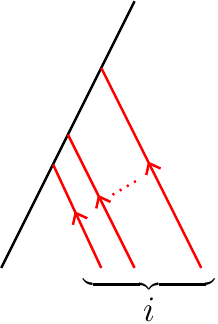}{.4}  \]
Note that the $(-1)^{N+1}$ can be removed via renormalisation of the generator in $\Hom( g^{\otimes N} \to \mathbf{1})$. However this results in a graph planar algebra embedding with more complicated coefficients.

We will show that given a cyclic near-group category $\mathcal{D}$, there exist parameters such that $\mathcal{C}(\mathbb{Z}_N, \omega, \vec{r})$ is a presentation for $\mathcal{D}$, and vice verse, that the Cauchy completion of the semisimplification of $\mathcal{C}(\mathbb{Z}_N, \omega, \vec{r})$ gives a near-group category for any parameters where the category is non-zero and unitary. Our motivation for this theorem is that it will allow us to query the existence of the categories $\mathcal{D}$ using graph planar algebra techniques. These graph planar algebra applications will be the focus of the next section.

For the remainder of this section let $\mathcal{D}$ be a unitary fusion category with $K_0(\mathcal{D}) \cong R(\mathbb{Z}_N, N)$. Further we will choose representatives of the isomorphism classes, and label them $\{   g : g\in \mathbb{Z}_N\} \cup \{\rho\}$.



We recall from Proposition~\ref{lem:subcat1} that $\mathcal{D}^\text{Inv} \simeq \operatorname{Vec}(\mathbb{Z}_N)$ as monoidal categories. The skein theory for $\operatorname{Vec}(\mathbb{Z}_N)$ is well-known (see e.g. \cite{agus}). There exists a generating morphism 
\[\att{ggen}{.4} \in \Hom_{\operatorname{Vec}(\mathbb{Z}_N)}( g^{\otimes N} \to \mathbf{1})\] satisfying the four $\operatorname{Vec}(\mathbb{Z}_N)$ relations as in Definition~\ref{def:NGskein}. By Proposition~\ref{lem:subcat1} we have the same generating morphism and relations in the category $\mathcal{D}$.

We also recall from Proposition~\ref{lem:subcat2} that $\mathbf{1}\oplus \rho \in \mathcal{D}$ has the structure of a Q-system. We thank Noah Snyder for informing us of general result on Frobenius algebras used in the proof of the following result.

\begin{prop}
Let 
\[      q_N:= \frac{1}{2}\sqrt{-2 +N + \sqrt{N^2 + 4N} + \sqrt{-16 + (N-2 + \sqrt{ N^2 + 4N   })^2               }     }  \]   
Then there exists a monoidal $\dag$-functor
\[  SO(3)_{q_N} \to \mathcal{D}        \]
sending the standard generating object of $SO(3)_{q_N}$ to $\rho$.
\end{prop}
\begin{proof}
We first note that $q$ is chosen so that $[3]_{q_N} = \delta$. By Proposition~\ref{lem:subcat2} we have that $\mathbf{1} \oplus \rho \in \mathcal{D}$ is a Frobenius algebra object with multiplication map $m: (\mathbf{1}\oplus \rho)\otimes (\mathbf{1}\oplus \rho) \to \mathbf{1}\oplus \rho$. Let $\iota : \rho \to \mathbf{1}\oplus \rho$ be an inclusion map, and define $\pi := \iota^\dag : \mathbf{1}\oplus \rho \to \rho$ the corresponding projection. Define
\[    \att{T}{.4} \quad:=\quad  \att{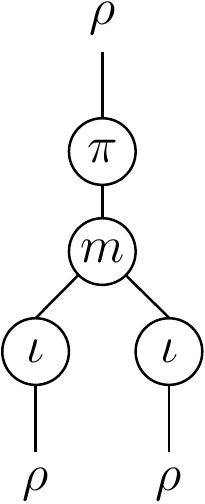}{.4}    \]
It follows by direct computation that this morphism satisfies the defining relations of $SO(3)_{q_N}$, and hence we have a canonical monoidal functor as in the statement of the proposition.
\end{proof}

As a consequence of the above proposition, we have that $\mathcal{D}$ contains $SO(3)_{q_N}$ skein theory. We now move on to the mixed relations.

The fusion rules of $\mathcal{D}$ give that there is an isomorphism 
\[  \att{isoRB}{.4} \in \Hom(\rho\otimes g \to \rho).  \]
The fact that $\mathcal{D}$ is unitary allows us to choose this isomorphism to be unitary, giving the relation
\[\att{isoRel1}{.3} =  \att{isoRel2}{.3}.\]
Note that we are free to rescale this isomorphism by an element of $U(1)$ while preserving the above relation.
As $\rho$ is assumed to be simple, there exists a scaler $\alpha \in \mathbb{C}$ such that
\[\att{gMix1}{.3} = \alpha \att{gMix2}{.3}.\]
By composing this relation with its dagger and using the $\operatorname{Vec}(\mathbb{Z}_N)$ relations, we see that $\alpha \in U(1)$. We now use up the $U(1)$ gauge freedom in the isomorphism to set $\alpha = (-1)^{N+1}$.

The fusion rules gives that there is a scalar $\omega\in \mathbb{C}$ such that  \[\att{swap1}{.3} = \omega   \att{swap2}{.3}.\]
The following is also due to Izumi.

\begin{prop}\cite{Izumi}
The scalar $\omega$ is a primitive $N$-th root of unity.
\end{prop}
\begin{proof}
This is shown by Izumi in \cite[Theorem 3.9]{Izumi}. More precisely he defines a function he denotes $\chi_a(b) :\mathbb{Z}_N \times \mathbb{Z}_N \to \mathbb{C}$ given by the scalar obtained by commuting an $a$-group strand connected on a $\rho$ strand, past a $b$-group strand. The cited theorem shows that this function is a non-degenerate symmetric bicharacter, and it follows that $\omega = \chi_1(1)$ is a primitive $N$-th root of unity.
\end{proof}

Finally from the fusion rules of $\mathcal{D}$ we have that $\dim\Hom(\rho^2 \to \rho) = N$. The following result shows that we have a natural basis for this space. We prove this result in slightly greater generality than is currently needed, as we will need to apply this result to two different categories.

\begin{lem}\label{lem:inde}
   Let $\mathcal{E}$ be a unitary tensor category with generating objects $\rho, g$, morphisms as in the defining morphisms of $\mathcal{C}(\mathbb{Z}_N, \omega , \vec{r})$, and relations as in the defining relations of $\mathcal{C}(\mathbb{Z}_N, \omega , \vec{r})$, sans the final mixed relation. Then
   \[ \left\{\att{ti}{.3}: 0 \leq i < N \right\} \subset \Hom_{\mathcal{E}}(\rho\otimes \rho \to \rho) \]
   are linearly independent.
\end{lem}
\begin{proof}
    Consider the $N\times N$ matrix $B$ of pairings with respect to the non-degenerate inner product on $\CC({\Z_N},\omega, \vec{r})$. That is
    \[  B_{i,j}  := \operatorname{tr}\left(\att{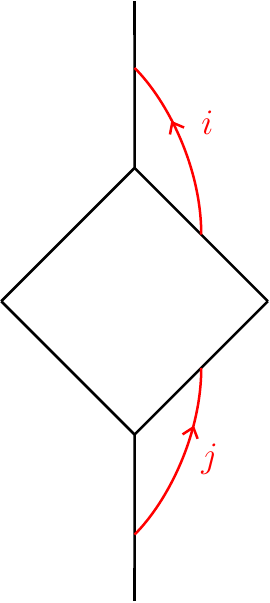}{.3}\right). \]
    Is is a straightforward exercise in skein theory to compute that the diagonal entries of $B$ are all $\delta$, and non-diagonal entries are all $\frac{\delta}{1-\delta}$. One now readily checks that the vectors
    \[
    e_1-e_2, \quad \dots, \quad e_{N-1}-e_N
    \]
    are linearly independent eigenvectors with eigenvalue $\delta +\frac{1}{\delta-1}$, and that
    \[
    e_1+\cdots+e_N
    \]
    is an eigenvector with eigenvalue $\delta-\frac{1}{\delta-1}$. Thus the matrix $B$ admits a basis of eigenvectors, and is therefore nonsingular. It follows as $\mathcal{E}$ is unitary that the morphisms
    \[ 
    \left\{\att{ti}{.3}: 0 \leq i < N \right\} \subset \Hom_{\CC({\Z_N},\omega, \vec{r})}(\rho\otimes \rho \to \rho) 
    \]
    are linearly independent as claimed.
\end{proof}

In the case of the category $\mathcal{D}$, the above lemma gives a basis for $\Hom_{\mathcal{D}}(\rho\otimes \rho \to \rho)$. It then follows that there exists $\vec{r} \in \mathbb{C}^N$ such that 
\[\att{T1S}{.3} = \sum_{i = 0}^{N-1}   \vec{r}_i \att{Ti}{.3}.\]

Finally by Schurs Lemma in the category $\mathcal{D}$, we have the relations
\[        \att{schur0}{.3}\quad = \quad 0,\qquad     \att{schur1}{.3} \quad = \quad 0          \]
for all $1\leq i < N$.


To summerise the results of this section up to this point, we have shown that $\mathcal{D}$ satisfies the same relations as the defining relations of our category $ \mathcal{C}(\mathbb{Z}_N, \omega, \vec{r})$ for some parameter choices. This implies the existence of a functor, which we now show descends to a dominant, full, and faithful functor out of the semisimplification of $ \mathcal{C}(\mathbb{Z}_N, \omega, \vec{r})$.

\begin{thm}
There exists a primitive $N$-th root of unity, and a $\vec{r} \in \mathbb{C}^N$ such that there is a dominant, full, and faithful monoidal $\dag$-functor
\[  \overline{  \mathcal{C}(\mathbb{Z}_N, \omega, \vec{r})    } \to \mathcal{D}.   \]
\end{thm}
\begin{proof}
The previous results of this section show that there is a monoidal $\dag$-functor $\mathcal{F}:\mathcal{C}(\mathbb{Z}_N, \omega, \vec{r}) \to  \mathcal{D}$ defined in the canonical way. As every object in $\mathcal{D}$ is a summand of $\rho^{\otimes n}$ for some $n$ the functor $\mathcal{F}$ is dominant.

By \cite[Corollary 1]{Caleb} we have that $\dim\Hom_{\mathcal{C}(\mathbb{Z}_N, \omega, \vec{r}) }(\mathbf{1}\to \mathbf{1} ) \leq1$, and it follows by Proposition~\ref{prop:descent} that $\mathcal{F}$ descends to a dominant faithful monoidal $\dag$-functor 
\[ \overline{\mathcal{F}}: \overline{  \mathcal{C}(\mathbb{Z}_N, \omega, \vec{r})    } \to \mathcal{D}.   \]

To show this functor is full we will apply Proposition~\ref{prop:ssFull}. We define in $\mathcal{C}(\mathbb{Z}_N, \omega, \vec{r})$ the projections
\[   p_{g^{i}} := \frac{1}{\delta}\att{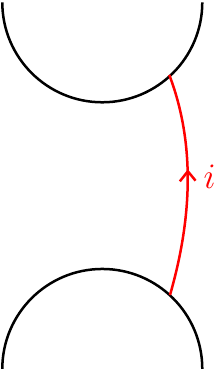}{.4}  \in    \operatorname{End}_{\CC({\Z_N},\omega, \vec{r})}(\rho \otimes \rho)  \]
and note the image of these projections under $\mathcal{F}$ in $\mathcal{D}$ project onto the simple objects $g^i$. As the trace of the $p_{g^{i}}$ is 1 in $\CC({\Z_N},\omega, \vec{r})$, we have that they survive in the semisimple quotient. In a similar fashion, we have that the identity morphism on $\rho$ in $\CC({\Z_N},\omega, \vec{r})$ is mapped under $\mathcal{F}$ to the identity on $\rho$, and hence projects onto $\rho$. The trace of this morphism is $\delta \neq0$, and so it also survives the semisimple quotient. As $\operatorname{Irr}(\mathcal{D}) = \{  g^i : 0\leq i < N   \}\cup \{\rho\}$ this implies fullness of $\overline{\mathcal{F}}$ via Proposition~\ref{prop:ssFull}.
\end{proof}

The preceding result shows that for any $\mathbb{Z}_N$-near group category $\mathcal{D}$, there exist parameter choices so that $\mathcal{C}(\mathbb{Z}_N, \omega, \vec{r})$ is a presentation for $\mathcal{D}$. To finish up we show that for any parameter choice, so long as $\CC({\Z_N},\omega, \vec{r})$ exists and the semi-simplification is unitary (which we will have for free in our GPA construction set-up), then the Cauchy completion of $\overline{\CC({\Z_N},\omega, \vec{r})}$ gives a $\mathbb{Z}_N$-near group category.
\begin{thm}\label{thm:karCom}
   Let $N \in \N$, $\zeta$ a $N$-th root of unity, and $\vec{r} \in \mathbb{C}^N$ be such that $\overline{\CC({\Z_N},\omega, \vec{r})}$ is non-trivial and unitary. Then 
   \[ K_0(\Ab(\overline{\CC({\Z_N},\omega, \vec{r})})   \cong R(\Z_N, N). \]
\end{thm}
\begin{proof}
    As $\overline{\CC({\Z_N},\omega, \vec{r}))}$ in unitary, we have that $\Ab(\overline{\CC({\Z_N},\omega, \vec{r})})$ is unitary, and in particular semisimple. Hence $\rho \otimes \rho$ decomposes as a direct sum of simples.

    From Lemma~\ref{lem:inde}, we have that $\dim\Hom_{\Ab(\overline{\CC({\Z_N},\omega, \vec{r})})}(\rho\otimes \rho\to \rho) \geq N$. By semisimplicity it follows that
    \[    \rho^{\oplus N} \subseteq  \rho\otimes \rho.\]

    By direct computation we have that
    \[   p_{g^{i}} := \frac{1}{\delta}\att{proji}{.4}  \in    \operatorname{End}_{\overline{\CC({\Z_N},\omega, \vec{r})}}(\rho \otimes \rho)  \]
    are pairwise non-isomorphic minimal projections. Another direct computation gives that
    \[\dim\Hom_{\Ab(\overline{\CC({\Z_N},\omega, \vec{r})})}(X \otimes X \to p_{g^{i}}) \geq 1. \]
    Together we have
    \[  \bigoplus_{g \in \Z_N} p_g \oplus \rho^{\oplus N} \subseteq \rho \otimes \rho.    \]
    A final direct computation gives that $\dim(p_{g^i}) = 1$. This shows that the quantum dimension of the both sides of the above equation agree, and hence it is an isomorphism (as any other simple summand would contribute positive quantum dimension by unitarity).

    Furthermore, we have that
    \[  \att{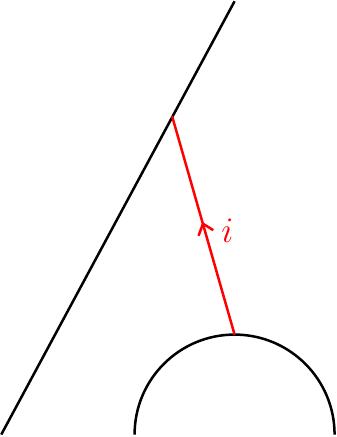}{.3} \in  \Hom_{\Ab(\overline{\CC({\Z_N},\omega, \vec{r})})}(\rho \otimes p_g \to \rho) \]
    is an isomorphism. Hence $\rho\otimes p_g \cong \rho$. It is now a straightforward induction to show that the only simple summands of $\rho^{\otimes n}$ are $\{ p_g : g\in \Z_N\}$ and $\rho$.
\end{proof}
   
\section{The Graph Planar Algebra Embeddings}\label{sec:gpa}

In this section we aim to give graph planar algebra embeddings for the 
$\mathcal{C}(\mathbb{Z}_N, \omega, \vec{r})$ skein theory developed in the previous section. 
The main goal of this section will be to produce elements in a certain coloured graph planar algebra 
that satisfy all of the defining relations of $\mathcal{C}(\mathbb{Z}_N, \omega, \vec{r})$ bar the final mixed relation. 
This final relation will be dealt with in the next section.

The first order of business is to determine a coloured graph whose graph planar algebra we will attempt to embed our diagrammatic category 
$\mathcal{C}(\mathbb{Z}_N, \omega, \vec{r})$ into. 
As this category has two object generators, $\rho$ and $g$, we will have to embed this category 
into a graph planar algebra of a graph with two colours. 
We will refer to the the $\rho$ colour of this graph as black, and the $g$ colour of this graph as red. 
As we know by Theorem~\ref{thm:karCom} that $\mathcal{C}(\mathbb{Z}_N, \omega, \vec{r})$ is a presentation 
for a fusion category with $R(\mathbb{Z}_N, N)$ fusion rules, the obvious choice is the $(\rho, g)$ fusion graph of 
$R(\mathbb{Z}_N, N)$ (drawn here in the case of $N=5$):

\[
    \begin{tikzpicture}[scale=1, vertex/.style={circle, draw, minimum size=5mm}]
        \node[vertex] (1) at (90:0)   {$\rho$};
        \node[vertex] (2) at (72:2)   {$\mathbf{1}$};
        \node[vertex] (3) at (144:2)  {$g^1$};
        \node[vertex] (4) at (216:2)  {$g^2$};
        \node[vertex] (5) at (288:2) {$g^3$};
        \node[vertex] (6) at (360:2)  {$g^4$};

        \draw (1) edge["5", loop, out = 300, in =210, looseness = 7]  (1) ;
        \draw (1) edge[red,loop] (1);
        \draw (1) edge  (2);
        \draw (2) edge[red,->] (3);
        \draw (1) edge  (3);
        \draw (3) edge[red,->] (4);
        \draw (1) edge  (4);
        \draw (4) edge[red,->] (5);
        \draw (1) edge  (5);
        \draw (5) edge[red,->] (6);
        \draw (1) edge  (6);
        \draw (6) edge[red,->] (2); 
    \end{tikzpicture}
\]

While we are guaranteed an embedding of $\mathcal{C}(\mathbb{Z}_N, \omega, \vec{r})$ (if it exists), 
into the graph planar algebra of this graph, this graph planar algebra is not amenable to computation due to 
the multiplicity on the black $\rho \to \rho$ edge. 
This results in a large amount of gauge freedom in our solution space. 
To obtain a graph more friendly to computation, we recall from Proposition~\ref{lem:subcat1} 
that the pointed subcategory of a near-group category always has trivial associator. 
This gives us that $A:= \oplus_{g \in \mathbb{Z}_N} g$ is an algebra object. 
We can then take our graph to be the $(\rho, g)$-coloured module fusion graph for the category of $A$-modules. 
An involved calculation using structure results of Izumi (which boils down to the scalar $\omega$ being a primitive $N$-th root of unity) gives the following graph.

\begin{defn}
    Let $N \in \mathbb{N}$, and $\mathcal{I} = \{ \text{black}, \text{red}    \}$. 
    We define the $\mathcal{I}$-coloured graph $\Gamma_N$ as the graph whose vertices are the labels $* , 0,1, \cdots, N-1$. 
    The black edges are
    \[  i \to j = \begin{cases}
    0 & \text{ if $i=*$ and $j=*$}\\
    1 & \text{ otherwise}\\
    \end{cases}         \]
    The red edges are all zero apart from
    \[ * {\color{red}\to} *, \quad  0 {\color{red}\to} 1 , \quad  1 {\color{red}\to} 2, \quad \cdots,\quad  N-2 {\color{red}\to} N-1, \quad \quad  (N-1) {\color{red}\to} 0.  \]
\end{defn}
 The black normalised Frobenius-Perron eigenvector of the graph $\Gamma_N$ is $\lambda_N := (z,1,,1,\dots,1)^T \in \C^N, \quad z := \frac{-N +\sqrt{N(N+4)}}{2}$, and the red normalised Frobenius-Perron eigenvector is identically $1$.

 For ease of notation we will assume the indices $\{0,1,\cdots, N-1\}$ are taken mod $N$, so that $N=0$. This will significantly simplify certain expressions. We will also formally enforce that $*+1 = *$, again for convenience of notation.

For an example $\Gamma_5$ is the graph in Equation~\eqref{eq:exGraph} in the introduction.

Defining the functor $\mathcal{J}: \mathcal{C}(\mathbb{Z}_N, \omega, \vec{r})\to GPA(\Gamma_N)$ on objects is simple. 
We map $\rho$ to black, and $g$ to red. Defining the images of the defining morphisms is much more involved.

\subsection{Embedding the $SO(3)_q$ vertex}
We begin by defining the image of the morphism 
\[ \att{T}{.25} \in \Hom_{ \mathcal{C}(\mathbb{Z}_N, \omega, \vec{r})   }(\rho\otimes \rho\to \rho)  \]
under the functor $\mathcal{J}$. 
Note that the image of this morphism lives in the black subcategory of the coloured graph planar algebra. 
This subcategory is simply the standard graph planar algebra on the graph $(\Gamma_N)_{\text{black}}$. 
To define the coordinates of this image, we recall the complex numbers
\[q_N :=\frac{1}{2}\sqrt{-2 +N + \sqrt{N^2 + 4N} + \sqrt{-16 + (N-2 + \sqrt{ N^2 + 4N   })^2               }     } \]
from Section~\ref{sec:skein}, and define the quantum integers
\[  [n]:= \frac{q_N^n - q_N^{-n}}{q_N - q_N^{-1}} .     \]

As the graph $\Gamma_N$ is multiplicity free, we can unambiguously use the shorthand 
\[(i\to k \to j, i \to j)\] 
to represent a pair of all black paths in $\Gamma_N$. 

\begin{defn}\label{def:taudef}
    Let $N \in \mathbb{N}$, and let $c(i,j,k) = 1 + \delta_{i=j} + \delta_{j=k} + \delta_{i = k} - \delta_{i=j=k}$ (i.e the function that counts how many of the inputs agree).
    We define $\tau   \in \Hom_{\operatorname{GPA}(\Gamma_N)}( \text{black}\otimes \text{black} \to \text{black})$ as the functional defined on basis elements by
    \[     \tau(i\to k \to j, i \to j) = \frac{1}{\sqrt{ \lambda_N(k)}}\begin{cases}   
    \sqrt{  \frac{[2]^3}{[3][4]}   } & c(i,j,k) = 1\\
    \sqrt{  \frac{[2]}{[3][4]}   } & c(i,j,k) = 2\\
    -\sqrt{  \frac{[4]}{[2][3]}   } & c(i,j,k) = 3
    \end{cases}   \]
\end{defn}
Note that this definition is symmetric with respect to any permutation of vertices $0,\dots,N-1$.

The above coefficients were obtained in a rather inelegant manner. We first explicitly solved in the case of $N=3$ for elements in $\operatorname{GPA}(\Gamma_3)$ satisfying the defining $SO(3)_{q_3}$ relations. Based on our experience with graph planar algebra embeddings from \cite{dan}, we thought it worthwhile to express the coefficients as products of half-powers of quantum integers in $q_3$. We then generalised these coefficients to the case of $q_N$. The remainder of this subsection will be verifying that this guess is indeed correct.

We recall that the presentation for $SO(3)_{q_N}$ we use is
\[
    \att{loop}{.25} = \delta, \quad 
    \att{T}{.25} = \att{Tdag}{.25}, \quad 
    \att{lolly}{.25}=0, \quad 
    \att{IH1}{.2} -\att{IH2}{.2} = \frac{1}{\delta-1}\left( \att{IH3}{.2} \quad -\quad \att{IH4}{.2} \right)
\]
where $\delta = [3]$.

\begin{thm}
    Let $N\in \mathbb{N}$. Then $\tau \in \Hom_{\operatorname{GPA}(\Gamma_N)}( \text{black}\otimes \text{black} \to \text{black})$ satisfies the defining $SO(3)_{q}$ relations.
\end{thm}
\begin{proof}
    Let $\tau(i,j,k)$ denote the coefficient of $\tau$ on the basis vector with path $(i \to k \to j,i \to j)$. 
    Define also $c(i,j)$ to be the coefficient of $coev_{\operatorname{black}}$ on the basis vector with path $(i \to j \to i, i)$. 
    Then to verify the (Lollipop) relation we check that for each $l\in \{*\}\cup \{0,\dots, N-1\}$ we have
    \[
    \sum_{k=1}^N \tau(l,l,k) c(l,k) = 0.
    \]
    When $l=*$ this holds vacuously, since $\tau(*,*,k)=0$ for all $k$. 
    By the symmetry of $\tau$ on the non-distinguished vertices, it remains only to check the $l=0$ case. 
    Walking through the definitions gives
    \begin{align*}
       \tau(l,l,*) c(l,*)+ \sum_{k=0}^{N-1} \tau(l,l,k) c(l,k) & = c(0,*) \sqrt{\frac{1}{[2] [4]}}  - c(0,0) \sqrt{\frac{[4]}{[2] [3]}}  
            + (N-1)c(0,1)\sqrt{\frac{[2]}{[3] [4]}}\\
        & =  \sqrt{\frac{\frac{-N +\sqrt{N(N+4)}}{2}}{[2] [4]}}  - \sqrt{\frac{[4]}{[2] [3]}}  + (N-1)\sqrt{\frac{[2]}{[3] [4]}}\\
        & = 0
    \end{align*}
    with the final equality readily verified by a computer.

    The rest of the verifications go through similarly. 
    See the Mathematica file so3Verify.nb attached to the arXiv submission of this note for all of the formal verifications.
\end{proof}

\begin{rmk}
    One notes in the process of verification that once the sums have been resolved, the resulting equations hold for all $N\in(0, \infty)$. 
    This suggests an interpolation of a functor $SO(3)_{q_N}\to GPA(\Gamma_N)$ is hiding somewhere. While the source of this hypothetical interpolation functor is clear and well-defined, the target remains a mystery to us. While the graph $\Gamma_\infty$ is well-defined, it is not locally finite, and so the corresponding graph planar algebra is not well-defined.
\end{rmk}

\subsection{The group embedding}

We now obtain an embedding for the generator
\[\att{ggen}{.4} \in \Hom_{ \mathcal{C}(\mathbb{Z}_N, \omega, \vec{r})   }( g^{\otimes N} \to \mathbf{1}) \]
in the graph planar algebra for $\Gamma_N$. As $g$ is mapped to red, we have that the embedding of the above generator is determined by it's coefficients on the pairs of paths
\[  (* \red{\to} *\red{\to} \cdots \red{\to} *, *) \quad \text{ and }  (i \red{\to}  i+1 \red{\to} \cdots \red{\to} i - 1 \red{\to} i, i  )  \quad \text{ for } 0\leq i \leq N-1.        \]
\begin{defn}
We define $\omega \in \operatorname{Hom}_{\operatorname{GPA}(\Gamma_N)}(  \textrm{red}^{\otimes N} \to \mathbf{1}  )$ as the functional defined on basis elements by 
\begin{align*}
\omega(* \red{\to} *\red{\to} \cdots \red{\to} *, *) &=1\\
\omega(i \red{\to}  i+1 \red{\to} \cdots \red{\to} i - 1 \red{\to} i, i) &=1 \quad \text{ for all } 1\leq i \leq N-1.
\end{align*}
\end{defn}
We now verify that this embedding satisfies the defining $\mathbb{Z}_N$ relations of $\mathcal{C}(\mathbb{Z}_N, \omega, \vec{r})$. The first two of these relations are immediately satisfied by the definition of the rigidity maps in $\operatorname{GPA}(\Gamma_N)$. The third relation and fourth relations reduce down to $1\times 1= 1$ and $\overline{1} = 1$.
and is also satisfied.

\subsection{Embedding the mixed generator}
To complete our embedding, we must find an element of $\GPA(\Gamma_N)$ for the image of the isomorphism
\[
    \att{isoRB}{.3}
\]
The image of this map in $\GPA(\Gamma_N)$ will live in $\Hom_{\GPA(\Gamma_N)}(\text{black}\otimes \text{red} \to \text{black}   )$, and hence will take a value on each pair of paths
\[  (i\to j \red{\to} j+1, i \to j+1)          \]
for $i,j \in \Gamma_N$ such that the above path exists.
\begin{defn}
Let $\Theta(i,j) \in \mathbb{C}$ for $i,j \in \Gamma_N$ such that $i\to j$. We define $\gamma_\Theta \in \operatorname{Hom}_{\operatorname{GPA}(\Gamma_N)}(  \textrm{black}\otimes \textrm{red} \to \textrm{black}  )$ as the functional defined on basis elements by 
\[\lambda_\Theta(i\to j \red{\to} j+1, i \to j+1):= \Theta(i,j).\]
\end{defn}
Note that from the definition of the graph $\Gamma_N$, the coefficient $\Theta(i,j)$ 
exists for all $i,j$ apart from $i = * = j$. 

We will now make some assumptions on our solution space. 
These assumptions will make solving for the scalars $\Theta(i,j)$ significantly easier, and will be justified by the fact that we can find solutions satisfying them. 
In practice these assumptions were obtained by solving small examples (i.e $N=2$ and $N=3$) and analysing the solution for patterns. 
The first is a symmetry assumption.

\begin{ass}\label{ass:sym}
We assume that the scalars $\Theta(i,j)$ satisfy the projective symmetry
\[   \Theta(i+1,j+1) = \omega  \cdot \Theta(i,j)          \]
for all $i,j \in \Gamma_N$.
\end{ass}

The second assumption concerns scalars which involve $i=*$ or $j=*$.

\begin{ass}\label{ass:star}
We assume that 
\[  \Theta(*,j) = \Theta(j,*) = \omega^{j}  \]
for all $0 \leq j <N$.
\end{ass}

Note that with these assumptions in play, the GPA embedding of the morphism  
\[\att{isoRB}{.3}\]
is entirely determined by the coefficients $\Theta(0, v)$ for $0\leq v\leq N-1$. It will be convenient to give these coefficients names.

\begin{defn}\label{def:fec}
Let $0\leq v\leq N-1$. We define
\[    P_v :=    \Theta(0,v)            \]
We will refer to these scalars as the free embedding parameters. For ease of notation for $0\leq v\leq N-1$ and $0\leq i \leq N-1$ we also define
\[Q_{v;i}:= \prod_{l=1}^i P_{v-l}.\]
\end{defn}
We also point out that due to Assumption~\ref{ass:sym}, we have the useful equation
\[   \Theta(v,w) = \omega^v  P_{w-v}   \]

\subsubsection{The normalisation relation}
Our first insight about $\Theta$ comes from the normalization choice on the red bigon. That is the relation
\[\att{isoRel1}{.3} =  \att{isoRel2}{.3}\]
For $\gamma_\Theta$ to satisfy this relation, we obtain the equations
\[1 = \Theta(i,j)\ol{\Theta(i,j)} = |\Theta(i,j)|^2 \quad \text{ for all $i,j \in \Gamma_N$}.\]
Therefore the values $\Theta(i,j)$ are unimodular.

\subsubsection{The absorbtion relation}

We consider the relation:

\[\att{gMix1}{.3} = (-1)^{N+1} \att{gMix2}{.3}\]

For $\gamma_\Theta$ to satisfy this relation, we have the following equations in terms of the free embedding parameters:
\[ \omega^{\frac{(N-1)N}{2}} = (-1)^{N+1}     \]
and
\begin{equation}\label{eq:absorb}   \prod_{i=0}^{N-1}P_i =   (-1)^{N+1}.             \end{equation}
The first of these equations is automatically satisfied as $\omega$ is a primitive $N$-th root of unity. Indeed if $N$ is odd, then 
\[     \omega^{\frac{(N-1)N}{2}} = (\omega^N)^{\frac{N-1}{2}}  = 1. \]
In a similar fashion if $N$ is even, then
\[  \omega^{\frac{(N-1)N}{2}} = (\omega^{\frac{N}{2}})^{N-1} = (-1)^{N-1}= -1.      \]
The Equation~\eqref{eq:absorb} is not automatic, and will have to solved for and verified. 
\subsubsection{The swap relation}

 Now recall the relation:
\begin{equation}
    \att{swap2}{.3} = \omega \att{swap1}{.3}
\end{equation}

We evaluate the functional $\gamma_{\Theta}$ on the basis vector
\[         (0 \to v \red{\to} v+1, 0 \red{\to} 1 \to v+1)    \]
on both sides of this relation. This results in the equations
\[  \Theta(1,v)\Theta(v,0)^{-1} = \omega \Theta(0,v) \Theta(v+1,0)^{-1} \qquad 0\leq v \leq N-1. \]
Using Assumptions~\ref{ass:sym} and Definition~\ref{def:fec} this gives the equations
\begin{equation}\label{eq:quad}   
    P_v P_{-v} = \omega P_{v-1}P_{-v-1} \qquad 0\leq v \leq N-1..                     
    \end{equation}
It turns out this quadratic system can be simplified to a linear system.
\begin{lem}\label{lem:linQuad}
In the case of $N$ odd, we have
\[   P_i = \omega^{i - \frac{N-1}{2}}P_{-i-1}       \qquad \text{for all }  0\leq i < N   \]
In the case of $N$ even, let $\tau$ be a choice of one of the two solutions to $\tau^2 = \omega$. Then we have
\[  P_i = \tau^{2i+ 1}P_{-i-1}  \qquad \text{for all }  0\leq i < N.  \]
\end{lem}
\begin{proof}
In the case of $N$ odd we begin by taking Equation~\eqref{eq:quad} in the case of $v = \frac{N+1}{2}$. This gives after simplification
\[ P_{\frac{N+1}{2}} P_{\frac{N-1}{2}} = \omega P_{\frac{N-1}{2}} P_{\frac{N-3}{2}}. \]
Hence $P_{\frac{N+1}{2}}= \omega P_{\frac{N-3}{2}}$, which is the desired equation in the case of $i = \frac{N+1}{2}$. We now proceed by induction to show the desired equation holds for $\frac{N+1}{2} < i < \frac{3N-1}{2}$ (note that $\frac{3N-1}{2} \equiv \frac{N-1}{2} \pmod N$).  For such an $i$, assume $ P_{i-1} = \omega^{i-1 - \frac{N-1}{2}}P_{-i}$. Consider Equation~\eqref{eq:quad} in the case of $v = i$. This gives after simplification
\[    P_{i} P_{-i} = \omega P_{i-1}P_{-i - 1}.            \]
Applying the assumption then gives
\[   P_{i} P_{-i} =  \omega^{i - \frac{N-1}{2}}P_{-i} P_{-i - 1}  .      \]
Hence $P_{i}  \omega^{i - \frac{N-1}{2}}P_{-i - 1}$ as desired. Thus we have shown \[P_i = \omega^{i - \frac{N-1}{2}}P_{-i-1}       \qquad \text{for all }\frac{N+1}{2}\leq i < \frac{3N-1}{2}, \]
which modulo $N$ is equivalent to the formula in the statement of the lemma.

The case of $N$ even follows in a similar fashion to the odd case. We begin with $v=0$ in Equation~\eqref{eq:quad} to obtain
\[  P_0^2 = \omega P_{-1}^2.   \]
Solving this gives $ P_0 = \tau P_{-1}$. We again proceed by induction, assuming that $P_{i-1} = \tau^{2i- 1}P_{-i}$. We then take Equation~\eqref{eq:quad} with $v=i$ to get 
\[    P_{i}P_{-i} = \omega   P_{i-1}P_{-i-1}   .  \]
Hence
\[   P_{i}P_{-i}   = \tau^2   \tau^{2i- 1}P_{-i}P_{-i-1}      \]
and so $P_{i} =   \tau^{2i+1}P_{-i-1}$ as desired.
\end{proof}

Now that Equation~\eqref{eq:quad} is solved, we have by Assumption~\ref{ass:sym} that $\gamma_{\Theta}$ satisfies the required equations on basis vectors containing paths which do not travel through the $*$ vertex. These remaining basis vectors are of the form
\[  (v \to * \red{\to} *, v \red{\to} v+1 \to *   ) \quad \text{and}\quad (*\to v \red{\to} v+1, * \red{\to} * \to v+1)   \]
for $0 \leq v <N$. Evaluating $\gamma_{\Theta}$ on these basis vectors yields the equations 
\[   \omega^{v+1} \omega^{ -v} = \omega \omega^v \omega^{-v} \quad \text{and}  \quad \omega^v \omega^{-v} = \omega \omega^v \omega^{-v-1}.       \]
Both of the equations hold, thus $\gamma_\Theta$ satisfies the swap relation precisely when the equations of Lemma~\ref{lem:linQuad} are satisfied.

\subsection{The Tadpole Relations}
We now bring our attention to the two families of tadpole relations. The easier of these from a GPA point of view are the relations

\[ \att{schur0}{.3}\quad = \quad 0\]
for $1\leq i \leq N-1$. For $\gamma_{\Theta}$ to satisfy these relations we note that the only admissible basis vector is
\[  (* \red{\to} *\red{\to} \cdots \red{\to} *, *)   .                \]
The coefficient of the embedding of the left hand side of the above relations on this basis vector is 
\[        \frac{1}{\sqrt{\lambda_*}} \sum_{v=0}^{N-1}  \Theta(v,*)^i = \frac{1}{\sqrt{\lambda_*}} \sum_{v=0}^{N-1}  (\omega^i)^v     \]
where we have used Assumption~\ref{ass:star}. As $\omega$ is a primitive $N$-th root of unity, we have that $\omega^i \neq 1$ for all $1\leq i \leq N-1$. It follows that $\sum_{v=0}^{N-1}  (\omega^i)^v = 0$, and hence $\gamma_{\Theta}$ already satsifies the above relations for all $i$.

The more complicated relations are the family
\[\att{schur1}{.3} \quad = \quad 0\]
for $1\leq i \leq N-1$. 
From these relations we will obtain equations in the free embedding variables $P_v$ (which we will later use to find explicit solutions). The admissible basis vectors for the left hand side are
\[   (v \red{\to} v+1\red{\to} \cdots \red{\to} (v+i), v\to (v+i)   )        \]
for $0\leq v \leq N-1$. By Assumption~\ref{ass:sym} it suffices to verify this relation only for $v=0$. This results in the equations
\[        \tau(0, i, *) \sqrt{\lambda_*} \prod_{l=0}^{i-1} \Theta(*,l) + \sum_{j=0}^{N-1} \tau(0,i,j)\prod_{l=0}^{i-1} \Theta(j,l)         \]
for $1\leq i \leq N-1$. After plugging in known values and simplifying, we obtain
\begin{equation*}\label{eq:tad}  \sqrt{N} \omega^{(i-1)i/2}   +    Q_{i;i}+   \omega^{i^2} Q_{0;i}    + [2]\sum_{j=1, j\neq i}^{N-1}\omega^{ij} Q_{i-j;i} = 0 \quad 1\leq i \leq N-1.   \end{equation*}
Hence $\gamma_\Theta$ satisfies the above tadpole relations precisely when the free embedding parameters satisfy Equation~\eqref{eq:tad}. Note that these are equations are degree $i$ polynomials in the embeddings variables $P_v$. In particular for $i=1$ we get a linear equation.
\subsection{The Change of Basis Relation}

We now move onto the most complicated relation
\begin{equation}\label{eq:COB}        \att{T1S}{.3} = \sum_{i = 0}^{N-1}   \vec{r}_i \att{Ti}{.3}.    \end{equation}
Currently we are in the situation where the coefficients $\vec{r}$ are unknown. However in the graph planar algebra, we are able to find a linear subsystem (in the free embedding variables) of equations required for the embedding to satisfy \eqref{eq:COB}. This allows us to solve for the values $\vec{r}$ in terms of the free embedding variables.

To obtain equations involving the values $\vec{r}_i$, we evaluate the coefficient of the embedding of the left and right hand of Equation~\eqref{eq:COB} on the basis vector
\[  (0\to v \to *, 0\to *)           \]
for $0\leq v \leq N-1$. This results in the equations
\[  \tau(0,*, v-1)\omega^{v-1} P_{v-1} = \sum_i \vec{r}_i  \tau(0,*,v)\omega^{v i}  .       \]
Using the values of $\tau(i,j,k)$ from Definition~\ref{def:taudef} this yields the linear system:
\[
\begin{bmatrix}
        \frac{\omega^{N-1}[2]}{P_{N-1}} \\
        \frac{1}{[2]P_0}\\ \frac{\omega}{P_1} \\
        \vdots\\
        \frac{\omega^{N-2}}{{P}_{N-2}}
    \end{bmatrix} = \begin{bmatrix}  \omega^{-i j }     \end{bmatrix}_{0\leq i,j \leq N-1} \cdot  \vec{r}
\]
Inverting this linear system yields the following solution for $\vec{r}$ in terms of the free embedding parameters $P_i$:
\[  
    \vec{r} = \frac{1}{N}
    \begin{bmatrix} 
    1 & 1 & 1&\cdots & 1 \\
    1& \omega &\omega^2 &  \cdots & \omega^{N-1} \\
    1 & \omega^2 & \omega^4 & \cdots &  \omega^{N-2}  \\
    \vdots & \vdots & \vdots &\ddots & \vdots\\
    1 & \omega^{N-1} &\omega^{N-2} & \cdots & \omega
    \end{bmatrix} 
    \begin{bmatrix}
        \frac{\omega^{N-1}[2]}{P_{N-1}} \\
        \frac{1}{[2]P_0}\\ \frac{\omega}{P_1} \\
        \vdots\\
        \frac{\omega^{N-2}}{{P}_{N-2}}
    \end{bmatrix}.  
\]

Note that this solution for $\vec{r}$ was obtained by solving a small subsystem of the full system of equations required for the relation in Equation~\ref{eq:COB} to hold. Hence we must verify more equations. By the projective symmetry of Assumption~\ref{ass:sym}, it suffices to verify this relation on the basis vectors
\begin{align*}
(* \to 0 \to v, *\to v)   & \qquad 0\leq v \leq N-1\\
(0 \to * \to v, 0\to v)   & \qquad 0\leq v \leq N-1\\
(0 \to v \to w, 0\to w)   & \qquad 0\leq v,w \leq N-1.
\end{align*}

The first of these gives the equations:
\[ \tau(*, v, N-1)\omega^{v+1}Q_{-v;1} = \sum_{i=0}^{N-1} \tau(*,v-i, 0)\vec{r}_i \omega^{iv - i(i+1)/2} Q_{v;i}^{-1}   \]
The second gives:
\[  \tau(0,v,*)\omega^v = \sum_{i=0}^{N-1}  \tau(0,v-i,*) \vec{r}_i \omega^{-iv + i(i+1)/2} Q_{v;i}       \]
The final one gives
\[         \tau(0,w,v-1)\omega^w Q_{v-w;1}Q_{v;1}^{-1} = \sum_{i=0}^{N-1} \vec{r}_i \tau(0,w-i, v) \omega^{-iv} Q_{w;i}Q_{w-v;i}^{-1}    \]





In summary we now have a set of equations in the free embedding variables, a solution of which gives an embedding of $\mathcal{C}(\mathbb{Z}_N, \omega, \vec{r})$ into $\operatorname{GPA}(\Gamma_N)$. This is the main result of this section.

\begin{thm}\label{thm:embThm}
Let $N \in \mathbb{N}$, and $\omega$ a primitive $N$-th root of unity. Further, let $\{  P_i : 0\leq i \leq N-1\}$ be complex unimodular scalars satisfying the following equations:
\begin{align*}
P_{i}&=\begin{cases}
\omega^{i - \frac{N-1}{2}}  P_{-i-1}   &\text{ if $N$ odd} \\
\tau^{2i +1}  P_{-i-1}   &\text{ if $N$ even, where $\tau^2 = \omega$} \\
\end{cases}   && 0\leq i \leq N-1 \\
\prod_{i= 0}^{N-1} P_i &= (-1)^{N+1}\\
  0&=\sqrt{N} \omega^{(i-1)i/2}   +    Q_{i;i}+   \omega^{i^2} Q_{0;i}    + [2]\sum_{j=1, j\neq i}^{N-1}\omega^{ij} Q_{i-j;i}  && 0\leq i \leq N-1\\
 \tau(*, v, N-1)\omega^{v+1}Q_{-v;1} &= \sum_{i=0}^{N-1} \tau(*,v-i, 0)\vec{r}_i \omega^{iv - i(i+1)/2} Q_{v;i}^{-1}&& 0\leq v \leq N-1\\
  \tau(0,v,*)\omega^v &= \sum_{i=0}^{N-1}  \tau(0,v-i,*) \vec{r}_i \omega^{-iv + i(i+1)/2} Q_{v;i} && 0\leq v \leq N-1\\
       \tau(0,w,v-1)\omega^w Q_{v-w;1}Q_{v;1}^{-1} &= \sum_{i=0}^{N-1} \vec{r}_i \tau(0,w-i, v) \omega^{-iv} Q_{w;i}Q_{w-v;i}^{-1} && 0\leq v,w \leq N-1
\end{align*}
where we recall $Q_{v;i}:= \prod_{l=1}^i P_{v-l}$, the function $\tau$ as in Definition~\ref{def:taudef}, and
\[  
    \vec{r} = \frac{1}{N}
    \begin{bmatrix} 
    1 & 1 & 1&\cdots & 1 \\
    1& \omega &\omega^2 &  \cdots & \omega^{N-1} \\
    1 & \omega^2 & \omega^4 & \cdots &  \omega^{N-2}  \\
    \vdots & \vdots & \vdots &\ddots & \vdots\\
    1 & \omega^{N-1} &\omega^{N-2} & \cdots & \omega
    \end{bmatrix} 
    \begin{bmatrix}
        \frac{\omega^{N-1}[2]}{P_{N-1}} \\
        \frac{1}{[2]P_0}\\ \frac{\omega}{P_1} \\
        \vdots\\
        \frac{\omega^{N-2}}{{P}_{N-2}}
    \end{bmatrix}.  
\]

Then there exists a non-trivial $\dag$-functor $\mathcal{C}(N, \omega, \vec{r}) \to GPA(\Gamma_N)$.
\end{thm}

As a consequence, we have the following existence result for near-group fusion categories.

\begin{cor}\label{cor:main}
Let $N \in \mathbb{N}$, $\omega$ a primitive $N$-th root of unity, and let $\{  P_i : 0\leq i \leq N-1\}$ be complex unimodular scalars satisfying the equations of Theorem~\ref{thm:embThm}. Then $\operatorname{Ab}( \overline{\mathcal{C}(N, \omega, \vec{r})})$ is a unitary fusion category with fusion ring $R(\mathbb{Z}_N,N)$, where $\vec{r}$ is an in Theorem~\ref{thm:embThm}.
\end{cor}
\begin{proof}
By Theorem~\ref{thm:embThm} we have that there is a non-trivial $\dag$-functor $\mathcal{C}(N, \omega, \vec{r}) \to GPA(\Gamma_N)$. It follows from \cite[Proposition 2.39]{EM-S} that $\overline{\mathcal{C}(N, \omega, \vec{r})}$ is non-trivial and unitary. We can now apply Theorem~\ref{thm:karCom} to obtain 
\[ K_0(\Ab(\overline{\CC({\Z_N},\omega, \vec{r})})   \cong R(\Z_N, N). \]
\end{proof}
\section{Examples}

In this section we find solutions to the equations of Theorem~\ref{thm:embThm} for several examples of odd $N$. This in turn constructs the corresponding near-group categories. This recovers existence results of Evans-Gannon \cite{E-G}, but with an alternate construction method. We will list our solutions to the free embedding variables $P_i$ for $0 \leq i \leq \frac{N+1}{2}$. The remaining variables are then immediate from the linear equation $ P_i = \omega^{i - \frac{N-1}{2}}P_{-1-i}$.

We define $\zeta_n := e^{2\pi i \frac{1}{n}}$. We present our solutions as algebraic extensions of the cyclotomic fields $\mathbb{Q}( \zeta_N, \zeta_{N+4})$. The quantity $[2]
 = \sqrt{N} + \sqrt{N+4}$ will play a key role in our solutions. When $N \equiv 1 \pmod 4$ we have that $[2]\in\mathbb{Q}(  \zeta_N, \zeta_{N+4})$. However when $N \equiv 3 \pmod 4$ we have $[2]\not\in\mathbb{Q}( \zeta_N, \zeta_{N+4})$. In this case we must extend to the larger cyclotomic field $\mathbb{Q}( \zeta_N, \zeta_{N+4}, i)$ to have $[2]$ in the field. For ease of notation we recall $\delta =\frac{N + \sqrt{N(N+4)}}{2}$, which is always in the field $\mathbb{Q}( \zeta_N, \zeta_{N+4})$.

We do not include the full verification of our solutions in this note. As our solutions are in explicit extensions of cyclotomic fields, the required equations can be computer verified quickly by reducing them modulo our defining (degree 2 and 4) polynomials. We include the computer verifications in Mathematica files attached to the arXiv submission of this note.  

We note that in order to verify that our free embedding variables are unimodular, we need to have an algebraic expression for the complex conjugate of our field generator $\alpha$. All of our defining polynomials have real coefficients, and so the complex conjugate of $\alpha$ is also an algebraic conjugate. In the degree two case this immediately gives the required algebraic expression for the conjugate. In the degree four case, the defining polynomial has constant term 1, and the results of \cite{numb-theory} guarantee the existence of at least one unimodular root of this polynomial. The real coefficients of the polynomial give that the complex conjugate of this unimodular root is also a root, and then the constant 1 term gives that the final two roots are also unimodular. It then follows that the complex conjugate of $\alpha$ in this case is its algebraic inverse.
\subsection{The Case of $N=3$}
For $N=3$ and $\omega = \zeta_3$ we find a solution over a degree 2 extension of the field 
\[ \mathbb{Q}( \zeta_{3},\zeta_{7}, i).   \]
This degree two extension is defined via the quadratic
\[  \alpha^2+\frac{1}{1+\delta} \cdot  \alpha +\frac{1}{1+\delta}.   \]
Our solution in this case is then
\begin{align*}
P_0 &:= \omega^2 [2] \alpha \\
P_1 &= \omega \left(    \frac{3}{2} +\frac{\mathbf{i}\sqrt{7}}{2}  + \omega \mathbf{i}\sqrt{7}+\alpha \right)
\end{align*}
\subsection{The Case of $N=5$}
For $N=5$ and $\omega = \zeta_5^2$ (we include this case instead of the $\omega = \zeta_{5}$ case as the solution for $\omega = \zeta_{5}$ is somewhat degenerate) we find a solution over a degree 2 extension of the field 
\[ \mathbb{Q}( \zeta_{5},\zeta_{9}).   \]
This degree two extension is defined via the quadratic
\[  \alpha^2 + \left( \frac{1}{2}(1 - \ii \sqrt{3})  - \ii \sqrt{3} \omega - \ii \sqrt{3} \omega^2 \right) \cdot   \alpha + 1.   \]
Our solution in this case is then
\begin{align*}
P_0 := \omega \cdot &\alpha \\
P_1 = \omega^4 \Biggl( &\left( \frac{1}{2}(-3 + \ii \sqrt{3}) + \ii \sqrt{3} \omega + (2 + \ii \sqrt{3})\omega^2 + 2\omega^3   \right) \\
+&\left( \frac{1}{2}(-1 + \ii \sqrt{3}) + \ii \sqrt{3} \omega + (1 + \ii \sqrt{3})\omega^2 + \omega^3   \right) \cdot \alpha \Biggl)\\
P_2 = \omega^2 \Biggl(& \left( \frac{1}{2} \left(1+i \sqrt{3}\right) \omega^3+\frac{1}{2} \left(1+3 i \sqrt{3}\right) \omega ^2+2 i \sqrt{3} \omega +i \sqrt{3}-1   \right) \\
+&\left( \frac{1}{2} \left(5+i \sqrt{3}\right) \omega ^3+\frac{1}{2} \left(5+3 i \sqrt{3}\right) \omega ^2+2 i \sqrt{3} \omega +i \sqrt{3}-1  \right) \cdot \alpha \Biggl)
\end{align*}


\subsection{The Case of $N=7$}
For $N=7$ and $\omega = \zeta_7$ we find a solution over a degree 2 extension of the field 
\[ \mathbb{Q}( \zeta_{7},\zeta_{11},i).   \]
This degree two extension is defined via the quadratic
\[  \alpha^2 +\frac{1}{1+\delta}\left(-\omega ^5+\frac{1}{2} \left(-5+i \sqrt{11}\right) \omega ^4+\frac{1}{2} \left(-5-i \sqrt{11}\right) \omega ^3-\omega ^2-2\right)\cdot   \alpha  + \frac{1}{1+\delta}.   \]
Our solution in this case is then

\begin{align*}
P_0 := \omega^5 [2] \cdot &\alpha \\
P_1 = \omega^2 \Biggl( &\left( \omega ^5+\frac{1}{2} \left(-5-i \sqrt{11}\right) \omega ^4+\frac{1}{2} \left(-5+i \sqrt{11}\right) \omega ^3+\omega ^2-4   \right) \\
+&\left(\left(-2-i \sqrt{11}\right) \omega ^5+\left(3-i \sqrt{11}\right) \omega ^4+3 \omega ^3-2 \omega ^2-i \sqrt{11} \omega +\frac{1}{2} \left(5-i \sqrt{11}\right) \right) \cdot \alpha \Biggl)\\
P_2 = \omega^6 \Biggl(& \left(\frac{1}{2} \left(1-i \sqrt{11}\right) \omega ^5+\left(-1-i \sqrt{11}\right) \omega ^4-\omega ^3+\frac{1}{2} \left(1-i \sqrt{11}\right) \omega ^2-i \sqrt{11} \omega +\frac{1}{2} \left(-5-i \sqrt{11}\right)   \right) \\
+&\left( \frac{1}{2} \left(-1+i \sqrt{11}\right) \omega ^5+i \sqrt{11} \omega ^4+\frac{1}{2} \left(-1+i \sqrt{11}\right) \omega ^2+i \sqrt{11} \omega +\frac{1}{2} \left(-3+i \sqrt{11}\right)  \right) \cdot \alpha \Biggl)\\
P_3 = \omega^3 \Biggl(& \left( \frac{1}{2} \left(-3+3 i \sqrt{11}\right) \omega ^5+\frac{1}{2} \left(1+3 i \sqrt{11}\right) \omega ^4+\frac{1}{2} \left(1+i \sqrt{11}\right) \omega ^3+\frac{1}{2} \left(-3+i \sqrt{11}\right) \omega ^2+2 i \sqrt{11} \omega +i \sqrt{11}+2 \right) \\
+&\left(\frac{1}{2} \left(-7-3 i \sqrt{11}\right) \omega ^5+\frac{1}{2} \left(1-3 i \sqrt{11}\right) \omega ^4+\frac{1}{2} \left(1-i \sqrt{11}\right) \omega ^3+\frac{1}{2} \left(-7-i \sqrt{11}\right) \omega ^2-2 i \sqrt{11} \omega -i \sqrt{11} \right) \cdot \alpha \Biggl)
\end{align*}

\subsection{The Case of $N=9$}
For $N=9$ and $\omega = \zeta_9$ we find a solution over a degree two extension of the field 
\[ \mathbb{Q}( \zeta_{9}, \zeta_{13}).   \]
This degree two extension is defined via the quadratic
\[  \alpha^2 +\left(1-\frac{1}{2} \left(\sqrt{13}-3\right) \left(\omega ^8+\omega \right)\right)\cdot   \alpha  + 1.   \]
Our solution in this case is then
\begin{align*}
P_0 := \omega^2 \cdot &\alpha \\
P_1 = \omega^7 \Biggl( & \frac{1}{106} \left(\left(103-21 \sqrt{13}\right) \omega ^5+\left(380-96 \sqrt{13}\right) \omega ^4+\left(75 \sqrt{13}-277\right) \omega ^2+\left(277-75 \sqrt{13}\right) \omega -50 \sqrt{13}+220)   \right) \\
+&\frac{1}{106} \left(-8 \left(4 \sqrt{13}-7\right) \omega ^5-7 \left(9 \sqrt{13}-29\right) \omega ^4+\left(31 \sqrt{13}-147\right) \omega ^2+\left(147-31 \sqrt{13}\right) \omega -56 \sqrt{13}+204\right)  \cdot \alpha \Biggl)\\
P_2 = \omega^3 \Biggl(& \frac{1}{106} \left(-\left(\left(7 \sqrt{13}+1\right) \omega ^5\right)+\left(321-85 \sqrt{13}\right) \omega ^4+\left(78 \sqrt{13}-322\right) \omega ^2+\left(322-78 \sqrt{13}\right) \omega -52 \sqrt{13}+250  \right) \\
+&\frac{1}{106} \left(\left(160-46 \sqrt{13}\right) \omega ^5+\left(421-127 \sqrt{13}\right) \omega ^4+9 \left(9 \sqrt{13}-29\right) \omega ^2-9 \left(9 \sqrt{13}-29\right) \omega -107 \sqrt{13}+439\right)  \cdot \alpha \Biggl)\\
P_3 = \omega^8 \Biggl(&\frac{1}{106} \left(\left(172-68 \sqrt{13}\right) \omega ^5+\left(597-167 \sqrt{13}\right) \omega ^4+\left(99 \sqrt{13}-425\right) \omega ^2+\left(425-99 \sqrt{13}\right) \omega -119 \sqrt{13}+407\right) \\
+&\frac{1}{106} \left(\left(146-38 \sqrt{13}\right) \omega ^5+\left(410-98 \sqrt{13}\right) \omega ^4+12 \left(5 \sqrt{13}-22\right) \omega ^2+\left(264-60 \sqrt{13}\right) \omega -93 \sqrt{13}+335\right) \cdot \alpha \Biggl)\\
P_4 = \omega^4 \Biggl(&\frac{1}{106} \left(7 \left(9 \sqrt{13}-29\right) \omega ^5+\left(129 \sqrt{13}-451\right) \omega ^4+\left(248-66 \sqrt{13}\right) \omega ^2+\left(66 \sqrt{13}-248\right) \omega +97 \sqrt{13}-395\right) \\
+&\frac{1}{106} \left(\left(43 \sqrt{13}-115\right) \omega ^5+\left(189 \sqrt{13}-715\right) \omega ^4+\left(600-146 \sqrt{13}\right) \omega ^2+2 \left(73 \sqrt{13}-300\right) \omega +115 \sqrt{13}-453\right) \cdot \alpha \Biggl)
\end{align*}

For our final two examples, the solutions are too large to reasonably write down in this paper. We will just list the field extension the solutions live in. The full solutions can be found in Mathematica files attached to the arXiv submission of this paper.

\subsection{The Case of $N=11$}
For $N=13$ and $\omega = \zeta_{11}$ we find a solution over a degree two extension of the field 
\[ \mathbb{Q}( \zeta_{11},\zeta_{15},i).   \]
This degree two extension is defined via the quadratic
\begin{align*}  \alpha^2 +\frac{1}{1+\delta}\Biggl( &    \frac{1}{2} \left(1-i \sqrt{3}\right) \omega ^{10}+          (\zeta_{15}+\zeta_{15}^4)  \omega ^9+   \frac{1}{2} \left(-1-i \sqrt{3}\right) \omega ^7 -\omega ^6-\omega ^5 +\frac{1}{2} \left(-1+i \sqrt{3}\right) \omega ^4\\&+ \omega ^2 (\zeta_{15}^{11} +\zeta_{15}^{14})+\frac{1}{2} \left(1+i \sqrt{3}\right) \omega -\frac{1}{2} \left(1+\sqrt{5}\right)\Biggr)\cdot   \alpha  + \frac{1}{1+\delta}.  
\end{align*}
\subsection{The Case of $N=13$}
For $N=13$ and $\omega = \zeta_{13}$ we find a solution over a degree four extension of the field 
\[ \mathbb{Q}( \zeta_{13},\zeta_{17}).   \]
Define
\begin{align*}
c_1 :=& \frac{1}{2} \left(\sqrt{17}-7\right) \omega ^{10}+\frac{1}{2} \left(\sqrt{17}-3\right) \omega ^9-\omega ^8+\left(\sqrt{17}-3\right) \omega ^7+\left(\sqrt{17}-3\right) \omega ^6-\omega ^5+\frac{1}{2} \left(\sqrt{17}-3\right) \omega ^4\\
&+\frac{1}{2} \left(\sqrt{17}-7\right) \omega ^3+\frac{1}{2} \left(\sqrt{17}-3\right)\\
c_2:=& \left(\sqrt{17}-2\right) \omega ^{11}+\left(9-2 \sqrt{17}\right) \omega ^{10}+\frac{1}{2} \left(5 \sqrt{17}-17\right) \omega ^8+\frac{1}{2} \left(9-\sqrt{17}\right) \omega ^7+\frac{1}{2} \left(9-\sqrt{17}\right) \omega ^6 \\
&+\frac{1}{2} \left(5 \sqrt{17}-17\right) \omega ^5+\left(9-2 \sqrt{17}\right) \omega ^3+\left(\sqrt{17}-2\right) \omega ^2-2 \sqrt{17}+13.
\end{align*}
The degree four extension is then defined by the polynomial
\[  \alpha^4 + c_1 \alpha^3 + c_2 \alpha^2 + c_1\alpha + 1.     \]

\printbibliography

\end{document}